\documentclass[12pt,reqno]{amsart}

\usepackage{amssymb,latexsym}

\usepackage{enumerate}
\allowdisplaybreaks
\usepackage[french,english]{babel}
\usepackage{amsmath}
\usepackage{graphicx}
\usepackage{amssymb}
\usepackage{bbm}
\usepackage{amsthm,mathtools}
\usepackage{ulem}
\usepackage{geometry}
\usepackage{tikz-cd}
\usepackage{mathrsfs}
\usepackage[colorinlistoftodos]{todonotes}
\usepackage{enumitem}
\usepackage{verbatim}
\usepackage[foot]{amsaddr}
\usepackage{dsfont}

\makeatletter

\@namedef{subjclassname@2010}{
	
	\textup{2020} Mathematics Subject Classification}

\makeatother
\newtheorem{thm}{Theorem}[section]
\newtheorem*{thm*}{Theorem}

\newtheorem{lem}[thm]{Lemma}
\newtheorem{pro}[thm]{Proposition}
\theoremstyle{definition}

\newtheorem*{xrem}{Remark}
\numberwithin{equation}{section}

\newcommand{\mbc}{\mathbb{C}}

\newcommand{\mbz}{\mathbb{Z}}
\newcommand{\mbr}{\mathbb{R}}

\newcommand{\mbn}{\mathbb{N}}

\newcommand{\mca}{\mathcal{A}}

\newcommand{\mce}{\mathcal{E}}

\newcommand{\mcp}{\mathcal{P}}

\newcommand{\mcl}{\mathcal{L}}

\newcommand{\mcm}{\mathcal{M}}

\newcommand{\ol}{\overline}

\usepackage{hyperref}
\hypersetup{hypertex=true,colorlinks=true,linkcolor=blue,anchorcolor=blue,citecolor=blue}
\usepackage{color}
\newcommand{\newabstract}[1]{%
	\par\bigskip
	\csname otherlanguage*\endcsname{#1}%
	\csname captions#1\endcsname
	\item[\hskip\labelsep\scshape\abstractname.]
}

\begin{document}

	\baselineskip=17pt

	\title[Large values of derivatives of the Riemann zeta function on vertical homogeneous progressions]{Large values of derivatives of the Riemann zeta function on vertical homogeneous progressions}

	\author{Qiyu Yang\textsuperscript{1}}
    \author{Shengbo Zhao\textsuperscript{2}}
	\address{1.School of Mathematics and Statistics, Henan Normal University, Xinxiang 453007, CHINA}
	\address{2.School of Mathematical Sciences, Key Laboratory of Intelligent Computing and Applications(Ministry of Education), Tongji University, Shanghai 200092, P. R. China}
	\email{qyyang.must@gmail.com}
	\email{shengbozhao@hotmail.com}

	\begin{abstract} 
	    In this paper, we establish lower bounds for the maximum of
        derivatives of the Riemann zeta function on vertical homogeneous progressions. When the real part $\sigma$ lies within a suitable range, we show that the discrete case has a similar order of magnitude to the continuous case, using the resonance method.
	\end{abstract}
	
	
	
	\maketitle

\section{Introduction}
The study of the Riemann zeta function $\zeta(s)$ has always been an important topic in analytic number theory, and the investigation of large values plays a crucial role in describing the value distribution of the Riemann zeta function. Here and throughout, we write $\log_k$ for the $k$-th iterated logarithm. On the critical line, i.e. $\sigma=1/2$, Soundararajan \cite{soundararajan2008extreme} introduced a highly effective technique, known as the resonance method, and proved the following lower bound
$$
\max_{T \le t \le 2T} \Big| \zeta\Big( \frac{1}{2} + it \Big) \Big| \geqslant \exp\bigg( (1 + o(1)) \sqrt{ \frac{\log T}{\log_2 T} } \bigg)
$$
as $T \to \infty$. This method originates Voronin's work  \cite{voronin1988lower}, which was later refined by Bondarenko and Seip \cite{bondarenko2017large,bondarenko2018argument}, who established the following bound for sufficiently large $T,$
$$
\max_{\sqrt{T} \le t \le T} \Big| \zeta\Big( \frac{1}{2} + it \Big) \Big| \geq \exp\bigg( (1+o(1)) \sqrt{ \frac{\log T \log_3 T}{\log_2 T} } \bigg).
$$
Notably, this result was later slightly improved by de la Bretèche and Tenenbaum \cite{tenen2019galsum} using optimized GCD sums. Specifically, the leading term in the exponent was improved from $(1+o(1))$ to $(\sqrt{2}+o(1))$. 
\par
For $\sigma=1$, Aistleitner, Mahatab and Munsch \cite{aistleitner2019extreme} established the following currently best known lower bound when $T \to \infty$,
$$
\max_{\sqrt{T} \le t \le T} |\zeta(1 + it)| \geq e^E (\log_2 T + \log_3 T - c)
$$
for some constant $c$, where $E$ is the Euler–Mascheroni constant. They developed the long resonance method, refining Soundararajan’s original approach. For the case where $1/2 < \sigma < 1$, there are also $\Omega$-results established for $\zeta(s)$ in \cite{aistleitner2016lower,bondarenko2018note,montgomery1977extreme} and so on.
\par
Large values of derivatives of $\zeta(s)$ have been extensively studied by Yang \cite{yang2022extreme,yang2024extreme} in recent years. In this paper, let $\mathbb{N}$ denote the set of all non-negative integers. As $\sigma=1/2$, in \cite{yang2022extreme} he proved that if $T$ is sufficiently large, 
$$\max _{t \in[0, T]}\Big|\zeta^{(j)}\Big(\frac{1}{2}+i t\Big)\Big| \geqslant \exp\bigg(\sqrt{2} \cdot\sqrt{\frac{\log T \log_3 T}{\log_2 T}}\bigg)$$
for all $j \in \mbn$. Furthermore, he also obtained extreme values of $\zeta^{(j)}(\sigma + it)$ in the critical strip $1/2 < \sigma <1$, see \cite[Theorem 2.(B)]{yang2022extreme}. The above results showed that large values of derivatives of $\zeta(s)$ have the same order of magnitude as the currently known extreme values of $\zeta(s)$ itself, up to constants in the exponent.
\par
Later, Yang \cite{yang2024extreme} showed that uniformly for all positive integers $j \le (\log_3 T)/(\log_4 T)$,
$$
\max_{T \le t \le 2T}|\zeta^{(j)}(1 + it)| \ge (Y_j + o(1))(\log_2 T)^{j+1},
$$
where $Y_j = \int_{0}^{\infty}u^j \rho(u) \mathrm{d}u$ and $\rho(u)$ is the Dickman function.
Including Yang's work, almost all above results utilize either Soundararajan’s resonance method or its modified versions. For more results and details about the resonance method, we recommend \cite{bondarenko2023dichotomy,darbar2025large,xumax2024extreme,qiyu2024large} and the references therein.
\par
Building on these works, we study large values of derivatives of $\zeta(s)$ on vertical arithmetic progressions, focusing particularly on homogeneous progressions. Our study is inspired by Li and Radziwiłł \cite{li2015theriemann} and Minelli and Sourmelidis \cite{minelli2025discrete}. Let $\mbr$ denote the set of all real numbers in this paper, and let $\alpha>0$ and $\beta \in \mbr$. Li and Radziwiłł showed that infinitely many positive integers $\ell$ satisfy
$$
\bigg| \zeta\Big( \frac{1}{2} + i(\alpha \ell + \beta) \Big) \bigg| \gg \exp\bigg( (1 + o(1)) \sqrt{ \frac{\log \ell}{6 \log_2 \ell} } \bigg).
$$
\par
Subsequently, Minelli and Sourmelidis continued their research on related problems. To obtain larger values, they restricted $\beta = 0$ and considered lower bounds for the maximum of $|\zeta(1/2+i\alpha \ell)|$ and $|\zeta(1 + i\alpha \ell)|$ on homogeneous progressions, see \cite[Remark 4.1]{minelli2025discrete} for reasons. Using approaches in \cite{bondarenko2017large,bondarenko2018argument} and \cite{aistleitner2019extreme}, they obtained following conclusions respectively. If $N\ge 1$ is sufficiently large, they showed that
$$
\max_{\sqrt{N} \leq \ell \leq N \log N} \bigg| \zeta\Big( \frac{1}{2} + i\alpha \ell \Big) \bigg| \geq \exp\bigg( \frac{1}{\sqrt{2}}\cdot \sqrt{ \frac{\log N \log_3 N}{\log_2 N} } \bigg)
$$
and
$$
\max_{\sqrt{N} \le \ell \le N}|\zeta(1+i \alpha \ell)| \ge e^E(\log_2 N + \log_3 N + O(1)).
$$
\par
Motivated by \cite{dong2025largevalue,minelli2025discrete,yang2022extreme,yang2024extreme}, we aim to obtain large values of derivatives of $\zeta(s)$ on vertical homogeneous progressions near the lines $\sigma=1/2$ and $\sigma =1$. Our results demonstrate that large values on homogeneous progressions achieve the same order of magnitude as those in the continuous case. Henceforth, we assume that $\ell \in \mbz$, where $\mbz$ is the set of all integers.

\begin{thm}\label{thm1}
    Let $j \in \mbn$ and fix $\alpha >0$. Let $A$ be any positive real number and $N$ be sufficiently large. Set     
    $\sigma_A \coloneqq 1/2 + A/(\log_2 N)$.
    We have
    \begin{equation*}
        \max_{\sqrt{N} \le \ell \le N} |\zeta^{(j)}(\sigma_A + i \alpha \ell)| \ge \exp\bigg((\lambda(A)+o(1)) \sqrt{\frac{\log N  \log_3 N }{\log_2 N }}\bigg),
    \end{equation*}
    where the leading term
    $$\lambda(A) = \frac{1}{\sqrt{2}(e-1)e^A}.$$
\end{thm}

Compared with \cite[Theorem 2.(A)]{yang2022extreme}, our Theorem \ref{thm1} shows that when $0< \sigma - 1/2 \ll (\log_2 N)^{-1}$, the large value of $\zeta^{(j)}(\sigma + i \alpha \ell)$ can reach a similar order of magnitude as those of $\zeta^{(j)}(1/2 + it)$ as $N \to \infty$, while the leading constant in the exponent is slightly smaller.

\begin{xrem}
    In fact, we can obtain the large values over any interval $[N^\theta, N], 0 < \theta <1$ by an almost identical method. In this case, the leading term in the exponent $\lambda(A)=\lambda(A, \theta)$ will depend on $\theta$. For simplicity, we consider only the case $\theta = 1/2.$
\end{xrem}

Then, we state the following result. Compared with \cite[Theorem 3]{dong2025largevalue}, Theorem \ref{thm2} yields a same lower bound when $0< 1- \sigma \ll (\log_2 N)^{-1}$. This also indicates that the currently known large values in the continuous case can also be realized on vertical homogeneous progressions.

\begin{thm}\label{thm2}
    Let $\alpha >0$ be fixed. Let $A$ be any positive real number and $N$ be sufficiently large. Set     
    $\sigma^{\prime}_A \coloneqq 1 - A/(\log_2 N)$.
    We have
    \begin{equation*}
        \max_{\sqrt{N} \le \ell \le N} |\zeta^{(j)}(\sigma^{\prime}_A + i \alpha \ell)| \ge (D_j(A) + o(1))(\log_2 N)^{j+1}
    \end{equation*}
    uniformly for all positive integers $j \le (\log_3 N)/(\log_4 N),$
    where 
    $$D_j(A) = \int_{0}^{\infty}e^{Au}u^j \rho(u) \mathrm{d}u.$$
\end{thm}

\begin{xrem}
In \cite[Theorem 1]{yang2024extreme}, Yang showed that the number of positive integers $j$ for which $|\zeta^{(j)}(1+it)|$ takes larger values in the range $t \in [T,2T]$ is smaller compared to \cite[Theorem 1]{yang2022extreme}. In \cite{yang2024extreme}, he removed the term $-\log_3 T$ and improved the leading term from $e^{\gamma}\cdot j^j \cdot (j+1)^{-(j+1)}$ to $Y_j = D_j(0).$ Moreover, Dong and Wei \cite{dong2021large} further refined Yang's result in \cite{yang2022extreme}. A similar argument applies to $|\zeta^{(j)}(\sigma^{\prime}_A + i \alpha \ell)|$. By slightly modifying the proof in \cite[Section 3]{yang2022extreme}, specifically by adjusting the choices of $x$ and
$b$, we can almost obtain a result similar to \cite[Theorem 1]{yang2022extreme}. Although this result yields a smaller lower bound, it holds for more positive integers $j$. We do not pursue this further, as Theorem \ref{thm2} suffices to illustrate the point discussed in the introduction.  
\end{xrem}
\par
Finally, we introduce some notations. Let $p$ be a prime number, and let $\mbc$ denote the set of all complex numbers. Let $\varepsilon > 0$ be an arbitrarily small number. We remark that each occurrence of $\varepsilon$ may represent a different value. Finally, we denote the Fourier transform of a function $f \in L^1(\mathbb{R})$ as
$$\widehat{f}(\xi) := \int_{\mathbb R} f(x) e^{-2\pi i  \xi x} \mathrm{d}x.$$

\section{Preliminary lemmas}
In this section, we present several lemmas that will be used subsequently. We begin with the following asymptotic formula for derivatives of $\zeta(s)$.
\begin{lem}[\cite{yang2024extreme}, Lemma 1]\label{lemmaYangdaodao}
    Fix $\sigma_0 \in (0,1)$. If $T$ is sufficiently large, then uniformly for $\varepsilon >0$, $t \in [T,2T]$, $\sigma \in [\sigma_0+ \varepsilon, \infty)$ and all $j \in \mbn$, we have
    \begin{equation*}
        (-1)^j \zeta^{(j)}(\sigma + it) = \sum_{n \leqslant T} \frac{(\log n)^j}{n^{\sigma + it}} + O\Big( \frac{j!}{\varepsilon^j} \cdot T^{-\sigma + \varepsilon} \Big),
    \end{equation*}
    where the implied constant only depends on $\sigma_0$.
\end{lem}
The next lemma is a key tool the subsequent calculations, providing an effective upper bound for mean squares.
\begin{lem}[\cite{montgomery1971topics}, Lemma 1.4]\label{lemmaGallagher}
    If $0<W<V-2$ and $g \, \text{:} \, [W,V] \to \mbc$ is continuously differentiable, then we have
    \begin{equation*}
        \sum_{W+1 \le n \le V-1}|g(n)|^2 \ll \int_{W}^{V}|g(t)|^2 \mathrm{d}t + \Big(\int_{W}^{V}|g(t)|^2 \mathrm{d}t \cdot \int_{W}^{V}|g^{\prime}(t)|^2 \mathrm{d}t \Big)^{\frac{1}{2}}.
    \end{equation*}
\end{lem}
Then, we can obtain the following upper bound for the discrete mean square of $\zeta^{(j)}(s)$.
\begin{lem}\label{lemmameansquare}
    Fix $j \in \mbn, \alpha>0$ and $\sigma \in (1/2, 1)$. For any sufficiently large $N \ge 1$, then 
    \begin{equation*} 
        \sum_{\ell \le N}|\zeta^{(j)}(\sigma + i\alpha\ell)|^2 \ll N,
    \end{equation*}
    where the implied constant only depends on $\sigma$.
\end{lem}
\begin{proof}
    For fixed $j \in \mbn$ and $\sigma \in (1/2, 1)$, the following classical result on the second moments of $\zeta^{(j)}(\sigma+it)$ holds, due to Hadamard, Landau, and Schnee (see \cite[Eq. (14)]{yang2024extreme}),
    \begin{equation*}
        \int_{0}^{T}|\zeta^{(j)}(\sigma + it)|^2 \mathrm{d}t \sim \zeta^{(2j)}(2\sigma)T.
    \end{equation*}
    In Lemma \ref{lemmaGallagher}, we set $g(t) = \zeta^{(j)}(\sigma + i\alpha t)$. Then combining with the above relation, we complete the proof of this lemma via simple calculations.
\end{proof}

\section{Proof of the Theorem \ref{thm1}}
\subsection{Constructing the resonator}\label{secConstucting}
We write $\sigma = \sigma_A = 1/2 + A/(\log_2 N)$ in this section. We use the resonance method which improved by Bondarenko and Seip \cite{bondarenko2017large,bondarenko2018argument}. In fact, our construction is very similar to the one in \cite{chirre2019extreme} by Chirre. 
\par
Let $\gamma \in (0,(e-1)^{-1})$ be a parameter to be chosen later and $N$ be a large number, then we define
$$
\mcp \coloneqq \{p : e \log N \log_{2}N < p \leq \exp\left(\left(\log_2 N\right)^{\gamma}\right) \log N \log_{2}N \}.
$$
 Next, we define the multiplicative function $f(n)$ to be supported on the set of square-free numbers, with values for $p \in \mathcal{P}$ as
$$
f(p) = \bigg( \frac{(\log N)^{1 - \sigma} (\log_2 N)^{\sigma}}{(\log_3 N)^{1 - \sigma}} \bigg) \frac{1}{p^{\sigma} (\log p - \log_2 N - \log_3 N )}.
$$
Furthermore, if $p \notin \mathcal{P}, \,f(p)=0.$ 
\par
For  $k \in \left\{1, \cdots, \lfloor(\log_2 N)^{\gamma}\rfloor\right\},$ we define the set $\mathcal{P}_{k}$  of prime numbers as follows
$$\mathcal{P}_{k} \coloneqq \{p : e^k \log N \log_{2}N < p \leq e^{k+1} \log N\log_{2}N \}.$$
For real number $b \in (e-1, 1/\gamma)$ fixed, define
$$\mathcal{M}_{k} := \left\{ n \in \mathrm{supp}(f) : n \text{ has at least } \Delta_k := \frac{b (\log N)^{2 - 2\sigma}}{k^2 (\log_3 N)^{2 - 2\sigma}} \text{ prime divisors in } \mathcal{P}_{k} \right\}.$$ 
Let $\mathcal{M}^{\prime}_{k}$ be the set of integers from $\mathcal{M}_{k}$ that have prime divisors only in $\mathcal{P}_{k}$, then set
$$
\mathcal{M} := \mathrm{supp}(f) \setminus \bigcup_{k = 1}^{\lfloor(\log_2 N)^{\gamma}\rfloor} \mathcal{M}_{k}.
$$
\par
Clearly, $\mathcal{M}$ is divisor closed. Specifically, if $m^{\prime} \mid m \in \mathcal{M},$ then $m^{\prime} \in \mathcal{M}.$ According to the similar process as \cite[Lemma 8]{chirre2019extreme}, we have $|\mcm| \le N$. Set $N = \lfloor T^\kappa \rfloor$, $\kappa \in (0,1)$ and let $\mcl$ be the set of integers $l$ such that 
$$
\left[\left(1 + \frac{\log T}{T}\right)^l, \left(1 + \frac{\log T}{T}\right)^{l + 1}\right] \bigcap \mathcal{M} \neq \emptyset.
$$ 
Then, we let $m_l$ be the minimum of $\left[(1 + \log T/T)^l, (1 + \log T/T)^{l + 1}\right] \cap \mathcal{M} $ for all $l \in \mathcal{L}$, and
$$
r(m_l) := \Bigg( \sum_{n \in \mathcal{M},  (1 - \log T/T)^{l - 1} \leqslant n \leqslant (1 + \log T/T)^{l + 2}} f(n)^2 \Bigg)^{\frac{1}{2}}.
$$
\par
Finally, we define the following resonator 
$$
R(t) = \sum_{m \in \mcm^{\prime}}r(m)m^{-it}.
$$
Note that we have trivially 
\begin{equation}
    \label{Rtbound}
    |R(t)|^2 \leq R(0)^2 \leq |\mcm|\sum_{m \in \mcm^{\prime}}r(m)^2
\end{equation}
by the Cauchy-Schwarz inequality.

\subsection{Auxiliary propositions}\label{Secauxiliary}
Based on the resonator $R(t)$ constructed in Section \ref{secConstucting}, we present some auxiliary results in this section. Among these results, some will be used repeatedly in the proof of Theorem \ref{thm1}, while others serve as key ingredients.
\par
We set $\Phi(y):=e^{-y^{2}/2}$ as in \cite{bondarenko2017large}. Plainly, the Fourier transform $\widehat{\Phi}$ satisfies $\widehat{\Phi}(\xi)=\sqrt{2\pi}\Phi(\xi).$ Due to the exponential decay of $\Phi$, all sums and integrals associated with $\Phi$ that appear below are absolutely convergent. Hence, we may freely interchange the order of these sums and integrals.
\par
We will use the following bounds frequently.
\begin{pro}\label{proposition3.1}
    Let $R(t)$ and $\mcm^{\prime}$ be as in Section \ref{secConstucting}.
    If $T$ is sufficiently large, we have
    \begin{equation}\label{intRPhi}
        \int_{\mbr}|R(t)|^2\Phi \Big( \frac{t \log T}{T} \Big) \mathrm{d}t  \asymp \frac{T}{\log T}\sum_{m \in \mcm^{\prime}}r(m)^2,
    \end{equation}
    and
    \begin{equation}\label{intRprimePhi}
        \int_{\mbr}|R^{\prime}(t)|^2\Phi \Big( \frac{t \log T}{T} \Big) \mathrm{d}t \ll T(\log T)^3\sum_{m \in \mcm^{\prime}}r(m)^2.
    \end{equation}
\end{pro}
\begin{proof}
     These results can be immediately obtained by replacing $T$ by $T/\log T$ in \cite[Proposition 2.7]{minelli2025discrete}.
\end{proof}
Next, define 
\begin{align*}
    \mathcal{A}_{N} := \frac{1}{\sum_{n\in\mathbb{N}} f(n)^2} \sum_{n\in\mathbb{N}} \frac{f(n)}{n^{\sigma}} \sum_{d \mid n} f(d)d^{\sigma}, \nonumber 
\end{align*}
and since $f$ is the multiplicative function defined in Section \ref{secConstucting}, we can write $\mca_N$ as 
\begin{equation*}
    \mca_N = \prod_{p\in \mathcal{P}} \frac{1 + f(p)^2 + f(p)p^{-\sigma}}{1 + f(p)^2}.
\end{equation*}
Similar to \cite[pp. 5-7]{dong2025largevalue}, we have the following propositions.

\begin{pro}
    \label{prop3.1} Let $\gamma$ be as in Section \ref{secConstucting}. Then for some $\delta \in (0,1)$, we have
    $$
    \mathcal{A}_{N} \geq \exp\bigg((\delta\gamma+ o(1))\frac{\left(\log N  \right)^{1-\sigma} \left(\log_3 N\right)^{\sigma}}{\left(\log_2 N\right)^{\sigma}}\bigg), \quad N \to \infty.
    $$
\end{pro}

\begin{proof}
      Since $f(p) < \left( \log_3 N\right)^{\sigma-1}$ for $p \in \mathcal{P},$ we obtain
    $$
    \mathcal{A}_{N} =\exp \bigg((1+o(1)) \sum_{p \in \mathcal{P}} \frac{f(p)}{p^{\sigma}}\Big(\frac{p}{p+1}\Big)\bigg)=\exp \bigg((1+o(1)) \sum_{p \in \mathcal{P}} \frac{f(p)}{p^{\sigma}}\bigg).
    $$
    The remainder of the proof follows from \cite[Proposition 3.1]{dong2025largevalue}.
\end{proof}

\begin{pro}
        \label{prop3.2}
    We have
    $$\frac{1}{\sum_{n \in \mathbb{N}} f(n)^2} \sum_{\substack{n \in \mathbb{N} \\ n \notin \mathcal{M}}} \frac{f(n)}{n^{\sigma}} \sum_{d \mid n} f(d)d^{\sigma}=o\left(\mathcal{A}_{N}\right), \quad  N \to \infty.
    $$

\end{pro}

\begin{proof}
    We can immediately obtain this conclusion through an almost the same argument as in \cite[Proposition 3.2]{dong2025largevalue}, which is similar to \cite[Lemma 2]{bondarenko2017large}.
\end{proof}

\begin{pro}
    \label{prop3.3}
    We have
    $$\frac{1}{\sum_{n \in \mathbb{N}} f(n)^2} \sum_{n \in \mathbb{N}} \frac{f(n)}{n^{\sigma}} \sum_{\substack{d \mid n \\ d \le n/N^{\varepsilon} }} f(d)d^{\sigma}=o\left(\mathcal{A}_{N}\right), \quad  N \to \infty.
    $$
\end{pro}

\begin{proof}
    Since $\sigma = \sigma_A \ge 1/2$, the proof can then be completed by the same method as \cite[Lemma 3]{bondarenko2017large}.
\end{proof}

Combining Propositions \ref{prop3.1}, \ref{prop3.2}, and \ref{prop3.3}, and employing \cite[Lemma 6]{bondarenko2018argument} and \cite[Eq. (19)]{bondarenko2018argument}, we have the following important consequence. 

\begin{pro}
    \label{prop3.4}
    Let $R(t)$, $\mcm$ and $\mcm^{\prime}$ be as in Section \ref{secConstucting}. Assume that $|\mcm| = \lfloor T^{\kappa} \rfloor$, $0< \kappa <1$ and $a_n \ge 0$ for every $n$. If $T$ is sufficiently large, we have
    \begin{align*}
        \int_{\mbr}\sum_{n \le \alpha T}\frac{a_n}{n^{\sigma + it}}&|R(t)|^2\Phi \Big(\frac{t \log T}{T} \Big)  \mathrm{d}t \\
        &  \ge \frac{T}{\log T} \exp \bigg(\Big(\delta \gamma \kappa^{1-\sigma} + o(1)\Big)\frac{\left(\log T  \right)^{1-\sigma} \left(\log_3 T\right)^{\sigma}}{\left(\log_2 T\right)^{\sigma}} \bigg) \sum_{m \in \mcm^{\prime}}r(m)^2
    \end{align*}
    for $0 < \delta < 1$ and $0 < \gamma < (e-1)^{-1}$.
\end{pro}

\subsection{Proof of Theorem \ref{thm1}}
Let $R(t)$ be as in Section \ref{secConstucting} and $\Phi$ be as in Section \ref{Secauxiliary}. To apply the resonance method, we begin with defining the following two sums
$$
S_1 \coloneqq S_1 (R,N) = \sum_{\ell \in \mbz} |R(\alpha \ell)|^2 \Phi \Big(\frac{\ell \log N}{N} \Big)
$$
and 
$$
S_2 \coloneqq S_2 (R,N) = \sum_{\ell \in \mbz} \sum_{n \le \alpha N} \frac{(\log n )^j}{n^{\sigma + i \alpha \ell}}|R(\alpha \ell)|^2 \Phi \Big(\frac{\ell \log N}{N} \Big),
$$
where $j \in \mbn$. 
Then, we estimate $S_1$ and $S_2$, respectively.
\par
First, we establish effective bounds for $S_1$. Furthermore, the rapid decay of $\Phi$ implies 
$$
S_1 \ll \sum_{| \ell| \le N}|R(\alpha \ell)|^2 \Phi \Big(\frac{\ell \log N}{N}\Big) + R(0)^2.
$$
In Lemma \ref{lemmaGallagher}, we set $g(t) = R(\alpha t)\sqrt{\Phi(t\log N / N)} $. Then, invoking the Minkowski inequality, we can bound the sum on the right-hand side above by
\begin{align*}
    \int_{\mathbb{R}} |R(at)|^2 &\Phi\Big( \frac{t\log N}{N} \Big)  \mathrm{d}t + \left( \int_{\mathbb{R}} |R(at)|^2 \Phi\Big( \frac{t\log N}{N} \Big) \mathrm{d}t \int_{\mathbb{R}} |R'(at)|^2 \Phi\Big( \frac{t\log N}{N} \Big) \mathrm{d}t \right)^{\frac{1}{2}} \\
    &+ \left( \int_{\mathbb{R}} |R(at)|^2\Phi\Big( \frac{t\log N}{N} \Big)\mathrm{d}t \int_{\mathbb{R}} |R(at)|^2 \Phi\Big( \frac{t\log N}{N} \Big) \frac{t^2 (\log N)^4 }{N^4} \mathrm{d}t \right)^{\frac{1}{2}}.
\end{align*}
By utilizing \eqref{intRPhi} and \eqref{intRprimePhi} in Proposition \ref{proposition3.1}, and combining \eqref{Rtbound}, we obtain the following upper bound for $S_1$
\begin{equation}
    \label{S1upperbound}
    S_1 \ll N\log N \sum_{m \in \mcm^{\prime}}r(m)^2.
\end{equation}
\par
In the opposite direction, by expanding out $R|(\alpha \ell)|^2$ in $S_1$, we obtain
$$
S_1 = \frac{N}{\log N} \sum_{m,n \in \mathcal{M}} r(m)r(n) \sum_{\ell \in \mathbb{Z}} \widehat{\Phi}\Big( N \Big( \frac{\alpha}{2\pi} \log \frac{m}{n} - \ell \Big) \Big).
$$
Here, we employ the Poisson summation formula. Since $r$ and $\widehat{\Phi}$ are always positive, we can neglect the terms with $ \ell \neq 0$, obtaining that
$$
S_1 \ge \frac{N}{\log N} \sum_{m,n \in \mathcal{M}} r(m)r(n) \sum_{\ell \in \mathbb{Z}} \widehat{\Phi}\Big( N \Big( \frac{\alpha}{2\pi} \log \frac{m}{n}\Big) \Big).
$$
Plainly, by \eqref{intRPhi}
\begin{equation}
    \label{S1lowerbound}
    S_1 \ge \int_{\mbr}|R(t)|^2\Phi \Big( \frac{t \log N}{N} \Big) \mathrm{d}t \gg \frac{N}{\log N}\sum_{m \in \mcm^{\prime}}r(m)^2.
\end{equation}
To obtain the lower bound for $S_2$, we employ a process analogous to that used above. We have
\begin{align*}
    S_2 & \, = \frac{N}{\log N} \sum_{n \le \alpha N} \sum_{m,h \in \mathcal{M}} \frac{(\log n)^j}{n^\sigma}r(m)r(h) \sum_{\ell \in \mathbb{Z}} \widehat{\Phi}\Big( N \Big( \frac{\alpha}{2\pi} \log \frac{m}{hn} - \ell \Big) \Big) \\
    & \, \geq \frac{N}{\log N} \sum_{n \le \alpha N} \sum_{m,h \in \mathcal{M}} \frac{(\log n)^j}{n^\sigma}r(m)r(h)  \widehat{\Phi}\Big( N \Big( \frac{\alpha}{2\pi} \log \frac{m}{hn} \Big) \Big). \\
    & \, =  \int_{\mbr}\sum_{n \le \alpha N}\frac{(\log n)^j}{n^{\sigma + it}}|R(t)|^2\Phi \Big(\frac{t \log N}{N} \Big)  \mathrm{d}t.
\end{align*}
Due to Proposition \ref{prop3.4}, for all $n \ge 1$, we derive the following lower bound for $S_2$
\begin{equation}
    \label{S2lowerbound}
    S_2 \ge \frac{N}{\log N} \exp \bigg(\Big(\delta \gamma \kappa^{1-\sigma} + o(1)\Big)\frac{\left(\log N  \right)^{1-\sigma} \left(\log_3 N\right)^{\sigma}}{\left(\log_2 N\right)^{\sigma}} \bigg)\sum_{m \in \mcm^{\prime}}r(m)^2.
\end{equation}
\par 
We next apply Lemma \ref{lemmaYangdaodao} to obtain the desired large value. First, we notice that
\begin{align*}
    \mce_1 & \, \coloneqq \sum_{|\ell| < \sqrt{N}} \sum_{n \le \alpha N}\frac{(\log n)^j}{n^{\sigma + i \alpha \ell}} |R(\alpha \ell)|^2\Phi \Big(\frac{\ell \log N}{N} \Big) \\
    & \, \le R(0)^2 \sum_{|\ell| < \sqrt{N}} \bigg|\sum_{n \le \alpha N}\frac{(\log n)^j}{n^{\sigma+ i \alpha \ell}} \bigg| \\
    & \, \le N^\kappa \bigg(\sum_{|\ell| < \sqrt{N}} |\zeta^{(j)}(\sigma + i\alpha \ell)| + O\Big( \frac{j!}{\varepsilon^j} \cdot N^{-\sigma + \varepsilon} \Big) \bigg)\sum_{m \in \mcm^{\prime}}r(m)^2.
\end{align*}
Combining Lemma \ref{lemmameansquare} and the Cauchy-Schwarz inequality, we can bound $\mce_1$ by
\begin{equation}
    \label{mce1upper}
    \mce_1 \ll N^{\frac{1}{2}+\kappa}\sum_{m \in \mcm^{\prime}}r(m)^2.
\end{equation}
Furthermore, as $\Phi$ decays rapidly, it follows that
\begin{equation}
    \label{mce2upper}
    \mce_2 \coloneqq  \sum_{|\ell| > N} \sum_{n \le \alpha N}\frac{(\log n)^j}{n^{\sigma + i \alpha \ell}} |R(\alpha \ell)|^2\Phi \Big(\frac{\ell \log N}{N} \Big) = o \Big(\sum_{m \in \mcm^{\prime}}r(m)^2 \Big).
\end{equation}
Employing Lemma \ref{lemmaYangdaodao} again, we obtain that
\begin{align*}
    &\sum_{\sqrt{N} \le |\ell| \le N} \bigg((-1)^j \zeta^{(j)}(\sigma + i\alpha \ell)|R(\alpha \ell)|^2\Phi \Big(\frac{\ell \log N}{N} \Big) - \sum_{n \le \alpha N}\frac{(\log n)^j}{n^{\sigma + i \alpha \ell}} |R(\alpha \ell)|^2\Phi \Big(\frac{\ell \log N}{N} \Big)\bigg) \\
    & \, \ll R(0)^2 \sum_{\sqrt{N} \le |\ell| \le N}\frac{j!}{\varepsilon^j}(\alpha N)^{-\sigma+\varepsilon} \ll N^{1-\sigma + \kappa + \varepsilon}\sum_{m \in \mcm^{\prime}}r(m)^2.
\end{align*}
\par
Noting that 
$$
S_2 = \sum_{\sqrt{N} \le |\ell| \le N}\sum_{n \le \alpha N}\frac{(\log n)^j}{n^{\sigma + i \alpha \ell}} |R(\alpha \ell)|^2\Phi \Big(\frac{\ell \log N}{N} \Big) + \mce_1 + \mce_2, 
$$
we arrive at
$$
\sum_{\sqrt{N} \le |\ell| \le N} (-1)^j \zeta^{(j)}(\sigma + i\alpha \ell)|R(\alpha \ell)|^2\Phi \Big(\frac{\ell \log N}{N} \Big) = S_2 +O(N^{\frac{1}{2}+\kappa}+N^{1-\sigma + \kappa + \varepsilon})\sum_{m \in \mcm^{\prime}}r(m)^2.
$$
From \eqref{S1lowerbound}, we deduce that
\begin{equation}
    \label{maxzetaj}
    \max_{\sqrt{N} \le \ell \le N} |\zeta^{(j)}(\sigma + i\alpha \ell)| \ge \frac{S_2}{S_1} + O\Big(\frac{N^{\frac{1}{2}+\kappa}+N^{1-\sigma + \kappa + \varepsilon}}{N/\log N} \Big).
\end{equation}
Regarding the ratio $S_2$ and $S_1$, it follows from \eqref{S1upperbound} and \eqref{S2lowerbound} that
\begin{align}
        \label{S2S2ratio}
    \frac{S_2}{S_1}& \,\ge \frac{1}{(\log N)^2}\exp \bigg(\Big(\delta \gamma \kappa^{1-\sigma} + o(1)\Big)\frac{\left(\log N  \right)^{1-\sigma} \left(\log_3 N\right)^{\sigma}}{\left(\log_2 N\right)^{\sigma}} \bigg) \nonumber \\
    & \, \ge \exp \bigg(\Big(\delta \gamma \kappa^{1-\sigma}e^{-A} + o(1)\Big)\sqrt{\frac{\log N \log_3 N}{\log_2 N}} \bigg).
\end{align}
Setting $\kappa< 1/2$, the error term in \eqref{maxzetaj} is $o(1)$. Note that
$$
\Big(\frac{1}{2}\Big)^{1-\sigma}=\Big(\frac{1}{2}\Big)^{1-\sigma_A}=\frac{1}{\sqrt{2}}+o(1).
$$
Then, by choosing $\kappa$, $\delta$ and $\gamma$ to be sufficiently close, respectively, to $1/2$, $1$ and $(e-1)^{-1}$, and for sufficiently large $N$, we can combine \eqref{maxzetaj} and \eqref{S2S2ratio} to obtain
\begin{equation*}
        \max_{\sqrt{N} \le \ell \le N} |\zeta^{(j)}(\sigma + i \alpha \ell)| \ge \exp\bigg(\Big(\frac{1}{\sqrt{2}(e-1)e^A}+o(1)\Big) \sqrt{\frac{\log N  \log_3 N }{\log_2N }}\bigg).
\end{equation*}

\section{Proof of the Theorem \ref{thm2}}
\subsection{Constructing the resonator} \label{secconres}
In this section, we set $\sigma \coloneqq \sigma^{\prime}_A = 1 - A/(\log_2 N).$
As in \cite{soundararajan2008extreme}, let $\Phi \,\colon \mbr \to \mbr$ denote a smooth function compactly supported in $[1,2]$, with $0 \le \Phi(t) \le 1$ for all $t$, and $\Phi(t) =1 $ for $t \in [5/4, 7/4]$. Integrating by parts, we deduce that $\widehat{\Phi}(\xi) \ll_\nu |\xi|^{-\nu}$ for any $\nu \in \mbn$.   
\par
We define the function $r \, \colon \mbn \to \{0,1\}$ to be the characteristic function of a specific set $\mcm$, coincides with that in \cite[Section 6.1]{dong2023Onde}. Then, we obtain the following conclusion, which serves as a key ingredient in the proof of Theorem \ref{thm2}.
\begin{pro}[\cite{dong2023Onde}, Proposition 1]
     \label{pro4.1}
     If $N$ is sufficiently large, uniformly for all positive integers $j \le (\log_3 N)/(\log_4 N), $ we have
    \begin{equation*}
        \Bigg| \sum_{mk = n \leq \sqrt{N}} \frac{r(m)\overline{r(n)}}{k^\sigma} (\log k)^j \Bigg| \bigg/ \left( \sum_{n \leq \sqrt{N}} |r(n)|^2 \right) \geq (D_j(A) + o(1)) (\log_2 N)^{j + 1}.
    \end{equation*}
\end{pro}
\begin{proof}
    The proof can be found in \cite[Section 6.1]{dong2023Onde}, and is similar to that in \cite[Section 3]{yang2024extreme}.
\end{proof}
Set $M = \lfloor N^{1/2} \rfloor$. Finally, we define the resonator by
$$
R(t) = \sum_{m \le M}r(m)m^{-it}.
$$
\subsection{Proof of Theorem \ref{thm2}}
Let $\Phi$ and $R(t)$ be as in \ref{secconres},and define the following two sums
$$
G_1 \coloneqq G_1(R,N) =  \sum_{\ell \in \mbz} |R(\alpha \ell)|^2 \Phi \Big(\frac{\ell}{N} \Big),
$$
and 
$$
G_2 \coloneqq G_2(R,N) =\sum_{\ell \in \mbz} \sum_{n \le 2\alpha N} \frac{(\log n )^j}{n^{\sigma + i \alpha \ell}}|R(\alpha \ell)|^2 \Phi \Big(\frac{\ell}{N} \Big).
$$
\par
It follows from \cite[Eq. (2)]{soundararajan2008extreme} that the following bound holds for $G_1$
\begin{equation}
  \label{G1upperbound}
    G_1  =(1+o(1)) \int_{\mbr} |R(\alpha t)|^2 \Phi\Big( \frac{t}{N} \Big) \mathrm{d}t = N \widehat{\Phi}(0) (1 + o(1)) \sum_{m \leq M} |r(m)|^2
\end{equation}
\par
As before, we expand $|R(\alpha \ell)^2|$ and apply the Poisson summation formula to obtain that
\begin{align*}
    G_2 & \, = N \sum_{k \le 2\alpha N} \sum_{m,n \leq M} \frac{(\log k)^j}{k^\sigma}r(m)\ol{r(n)} \sum_{\ell \in \mathbb{Z}} \widehat{\Phi}\Big( N \Big(\ell - \frac{\alpha}{2\pi} \log \frac{n}{km}\Big) \Big) \\
    & \ge N \sum_{k \le 2\alpha N} \sum_{m,n \leq M} \frac{(\log k)^j}{k^\sigma}r(m)\ol{r(n)}\widehat{\Phi}\Big(-N\frac{\alpha}{2\pi} \log \frac{n}{km}\Big).
\end{align*}
Due to the rapid decay of $\Phi$, for the off-diagonal terms where $m \neq kn$, we get
$$
\widehat{\Phi}\Big(-N\frac{\alpha}{2\pi} \log \frac{n}{km}\Big) \ll N^{-2}
$$
by the fact that $M \le N^{1/2}$. Thus, 
\begin{equation}
    \label{G2lowerbound}
    G_2 \ge N \widehat{\Phi}(0) \sum_{mk=n \leq M}\frac{r(m)\ol{r(n)}}{k^\sigma}(\log k)^j + o(1)\sum_{m \le M}|r(m)|^2.
\end{equation}
Observe that
$$
\sum_{\ell \in \mbz}\frac{j!}{\varepsilon^j}N^{-\sigma+\varepsilon} |R(\alpha \ell)|^2 \Phi \Big(\frac{\ell}{N} \Big) \ll \frac{j!}{\varepsilon^j}N^{-\sigma+\varepsilon}G_1.
$$
Then, Lemma \ref{lemmaYangdaodao} yields that
\begin{align*}
    & \,\max_{\sqrt{N} \le \ell \le N}|\zeta^{(j)}(\sigma + i\alpha \ell)| \ge \frac{|G_2|}{G_1} + O\Big(\frac{j!}{\varepsilon^j}N^{-\sigma+\varepsilon} \Big)   \\
    & \, \ge \Big(1+O(N^{-1})\Big)\Bigg| \sum_{mk = n \leq \sqrt{N}} \frac{r(m)\overline{r(n)}}{k^\sigma} (\log k)^j \Bigg| \bigg/ \left( \sum_{n \leq \sqrt{N}} |r(n)|^2 \right) + O((\log_2 N)^j).
\end{align*}
For sufficiently large $N$, the error term can be estimated using Stirling’s formula for all positive integers $j \le (\log_3 N)/(\log_4 N)$. The proof of Theorem \ref{thm2} can then be completed immediately by applying Proposition \ref{pro4.1}.

\section*{Acknowledgments}
Qiyu Yang was supported by the Natural Science Foundation of Henan Province (Grant No. 252300421782).
	
	\bibliographystyle{siam}
    \bibliography{reference}
\end{document}